\documentclass[]{amsart}
\usepackage{amsmath,amssymb,mathrsfs,graphicx}
\usepackage{amsthm,amscd}
\usepackage{type1cm}
\usepackage{ascmac}
\usepackage{ifpdf}
\usepackage{color}
\begin{document}
\newcommand{\capi}[2]{\framebox[2.5cm]{\textgt{\normalsize #1}}}
\newcommand{\bvec}[1]{\mbox{\boldmath $#1$}}
\newcommand{\subd}{\mathrm{Sd}}
\newcommand{\val}{\mathrm{val}}
\newcommand{\VAL}{\mathrm{Val}}

\theoremstyle{definition}

\newtheorem{thm}{Theorem}[section]
\newtheorem{cor}[thm]{Corollary}
\newtheorem{lem}[thm]{Lemma}
\newtheorem{prop}[thm]{Proposition}
\newtheorem{defi}[thm]{Definition}
\theoremstyle{remark}
\newtheorem{rem}[thm]{Remark}
\newtheorem{exa}[thm]{Example}

\title[A Note on tropical curves and the Newton diagrams of PCS]{A Note on tropical curves and the Newton diagrams of plane curve singularities}
\author{Takuhiro TAKAHASHI}
\address{Mathematical Insititute, Tohoku University, Aoba, Sendai, Miyagi, 980-8578, Japan}
\email{sb3m17@math.tohoku.ac.jp}
\maketitle
\begin{abstract}

For a convenient and Newton non-degenerate singularity, 
the Milnor number 
is computed from the complement of its Newton diagram 
in the first quadrant, 
so-called Kouchnirenko's formula. 
In this paper, we consider tropical curves dual to subdivisions 
of this complement 
for a plane curve singularity and show the existence of a tropical curve 
satisfying a certain formula, which looks like a well-known formula for a 
real morsification due to A'Campo and Gusein-Zade.
\end{abstract}

\section{Introduction}

Tropical geometry is 
rapidly developing 
as a new study area in mathematics. 
In \cite{Mik_enumerative}, 
Mikhalkin counts nodal curves on toric surfaces, 
which is an epoch-making result in algebraic geometry. 
Though 
there are several studies concerning tropical curves 
corresponding to 
singular algebraic curves, 
see for instance 
\cite{Shustin, MMS, Ganor}, 
the study of relation 
between theory of tropical curves and 
the singularity theory is still underdeveloping 
in tropical geometry.

In singularity theory, 
the Newton diagram is an important tool to get 
information of singularities. 
Let 
$f$ 
be a polynomial of 
two variables over $\mathbb{C}$, 
and suppose that 
$f(0,0)=0$ 
and 
$f$ 
has an isolated singularity 
at $0=(0,0) \in \mathbb{C}^2$. 
Set 
$f(x,y)=\sum_{(i,j)}a_{ij}x^i y^j$, 
and 
$\Delta_f=\mathrm{Conv}(\{ (i,j) \in \mathbb{R}^2 ; a_{ij} \neq 0 \})$, 
where $\mathrm{Conv}(\cdot)$ is the convex hull. 
The convex hull $\Delta_f$ is called the \textit{Newton polytope} of $f$. 
Let 
$\Gamma_{-}(f)$ 
be the polyhedron defined by 
\[
\mathrm{Closure}\Bigl(
(\mathbb{R}_{\ge 0})^2 \setminus 
\mathrm{Conv}(\{ (i,j)+(\mathbb{R}_{\ge 0})^2 ; a_{ij} \neq 0 \})
\Bigr), 
\]
where 
$\mathrm{Closure}(\cdot)$ 
is the closure with usual topology of 
$\mathbb{R}^2$. 
The \textit{Newton boundary} 
$\Gamma(f)$ 
of $f$ 
is the union of compact faces of 
$\mathrm{Conv}(\{ (i,j)+(\mathbb{R}_{\ge 0})^2 ; a_{ij} \neq 0 \})$. 
The singularity 
$(f,0)$ 
is \textit{convenient} if 
$\Gamma_{-}(f)$ 
is compact. 
The singularity $(f,0)$ is \textit{Newton non-degenerate} 
if, for any face $\sigma$ in $\Gamma(f)$,  
the function 
$f^{\sigma}(x,y):=\sum_{(i,j)\in \sigma \cap \Delta_f}a_{ij}x^i y^j$
has no singularity in $(\mathbb{C}^*)^2$. 

Note that, 
for a convenient and Newton non-degenerate singularity, 
the Milnor number of 
$(f,0)$ 
is computed from $\Gamma_-(f)$, which is 
the celebrated theorem of Kouchnirenko \cite{Kouch}.
In this paper, 
we regard 
$\Gamma_{-}(f)$ 
as a part of polyhedrons 
obtained 
as 
a dual subdivision of a tropical curve 
and give a meaning of 
$\Gamma_{-}(f)$ 
from the viewpoint of tropical geometry. 

To state our result, 
we here introduce some terminologies in tropical geometry. 
Let $F$ be a polynomial of two variables over 
the field $K$ of convergent Puiseux series over $\mathbb{C}$. 
The \textit{tropical curve} 
$T_F$ 
is defined by the 
image of 
$F=0$ 
by the valuation map, 
which is a 1-simplicial complex in 
$\mathbb{R}^2$. 
The valuation of the coefficients of $F$ 
induces a subdivision 
$\subd(F)$ 
of the Newton polytope 
$\Delta_{F}$ 
and it is known that 
$\subd(F)$ 
is dual to $T_F$, 
which is so-called 
the \textit{Duality Theorem} (see \S 2). 

Now we consider a union of polygons 
corresponding to a part of polygons of $\subd(F)$.
Let $\Delta'$ be a sub-polyhedron of 
$\Delta_F$. 
A subset $S$ of $T_F$ is called 
the \textit{tropical sub-curve with respect to $\Delta'$} 
if $\Delta'$ is a union of sub-polyhedrons of 
$\subd(F)$ and 
$S$ is dual to the subdivision of 
$\Delta'$ 
induced by $\subd(F)$. 
We denote it by $T_F|_{\Delta'}$. 
Note that if $\Delta'=\Delta_F$ 
then $T_F|_{\Delta'}=T_F$. 
See Definition \ref{resttrop} for the precise definition of 
$T_F|_{\Delta'}$. 

For plane curve singularities, 
the real morsification due to 
A'Campo \cite{AC} and Gusein-Zade \cite{GZ} 
gives an explicit way to understand mutual positions 
of vanishing cycles. 
Our hope is that we can perform 
the same observation 
for tropical curves realized in 
$\Gamma_{-}(f)$. 
The main theorem in this paper 
asserts that 
we can see the vanishing cycles of $f=0$ on 
the tropical curve in 
$\Gamma_{-}(f)$. 
For a tropical sub-curve $T_F|_{\Delta'}$, 
let 
$v(T_F|_{\Delta'})$ 
denote the number of 4-valent vertices 
of $T_F|_{\Delta'}$ and 
$r(T_F|_{\Delta'})$ 
denote the number of regions bounded by 
$T_F|_{\Delta'}\subset \mathbb{R}^2$.

\begin{thm} \label{MainThm}\it
For any Newton non-degenerate and convenient isolated singularity 
$(f,0)$, 
there is a polynomial 
$F:=F_f \in K[z,w]$ 
such that 
$\Delta_F=\mathrm{Conv}(\Gamma_{-}(f))$ 
and 
$T_F|_{\Gamma_{-}(f)}$ 
satisfies 
\[
\mu(f)=v(T_F|_{\Gamma_{-}(f)})+r(T_F|_{\Gamma_{-}(f)}).
\]
\end{thm}

This result is an analogy of the following equality required for a real morsification: 
\[ \mu(f) = \delta (f_s) + r(f_s),  \]
where 
$f_s$ is a real morsification of $f$, 
$\delta(f_s)$ is the number of double points of $f_s=0$ 
in a previously fixed small neighborhood 
$U$ 
of the origin, 
and 
$r(f_s)$ is the number of bounded regions of 
$\{ f_s = 0\}|_{\mathbb{R}^2} \cap U$. 
As a corollary of Theorem~\ref{MainThm}
we have the equality $\delta(f)=v(T_F|_{\Gamma_{-}(f)})$, 
where 
$\delta(f)$ 
is the number of double points of 
$(f,0)$, 
see Corollary~\ref{Cor}. 

Remark that 
the polynomial 
$F$ 
in Theorem \ref{MainThm} is given by a patchworking polynomial 
asssociated with a subdivision of $\Delta_F$. 
A similar observation appears in a paper of Shustin \cite{Shu2}, 
where a real polynomial whose critical points 
has a given index distribution is constructed. 

We organize the paper as follows. 
In section 2, 
we introduce 
tropical curves and subdivisions of Newton polytopes 
induced from the valuation map 
and state the Duality Theorem. 
In section 3, 
we give the definition of tropical sub-curves 
and 
prove Theorem~\ref{MainThm}. 
Two examples will be given before the proof. 

\subsection*{Acknowledgement}
I am grateful 
to Masaharu Ishikawa 
and Takeo Nishinou for fruitful discussions. 
I would like to thank Eugenii Shustin 
for precious comments and telling me about his previous result. 
I also thank Nikita Kalinin for giving me interesting comments.

\section{Preliminaries}

First, 
we give some definitions about polytopes. 
A \textit{polygon} 
in $\mathbb{R}^2$ 
is the intersection 
of a finite number of half-spaces 
in $\mathbb{R}^2$ 
whose vertices are 
contained in the lattice 
$\mathbb{Z}^2 \subset \mathbb{R}^2$. 
A polygon is called a \textit{polytope} if 
it is compact. 
In this paper, 
a \textit{polyhedron} means 
a union of polytopes 
which is 
connected and compact. 
Thus a polyhedron is not 
convex generally. 
A subset in a polyhedron 
is a \textit{sub-polyhedron} 
if it is a polyhedron 
as a subset of $\mathbb{R}^2$. 
In particular, 
if a sub-polyhedron is a polytope 
then the sub-polyhedron is called a \textit{sub-polytope}.

Let 
$K:=\mathbb{C}\{\{ t \}\}$ 
be the field of convergent Puiseux series over $\mathbb{C}$, 
and denote the usual non-trivial valuation on 
$K$ by $\val : K^{*} \to \mathbb{R}$, 
that is, 
\[ 
\val : K^* \longrightarrow \mathbb{R} \ ;\  
\sum_{k=k_0}^{\infty}b_{k}t^{\frac{k}{N}} \longmapsto -\frac{k_0}{N},
\]
where $K^{*}:=K \setminus \{0\}$ 
and $b_{k_0} \neq 0$. 

For a reduced polynomial 
\[ 
F(z,w)=\sum_{(i,j) \in \Delta_F \cap \mathbb{Z}^2}c_{ij} z^i w^j 
\in K[z,w] 
\]
over $K$, 
we denote by 
$\mathrm{Supp}(F)$ 
and 
$\Delta_F$ 
the \textit{support} of $F$ 
and 
the \textit{Newton polytope} of $F$, respectively, 
that is, 
$\mathrm{Supp}(F)
:=\{ (i,j) \in \mathbb{R}^2 ; c_{ij} \neq 0 \}$ 
and 
$\Delta_F
=\mathrm{Conv}(\mathrm{Supp}(F)) \subset \mathbb{R}^2$.
Throughout this paper, 
we assume that 
the Newton polytope $\Delta_F$ of $F$ is 2-dimensional. 

The valuation map 
$\VAL : (K^*)^2 \to \mathbb{R}^2$
is defined by 
$(z,w) \mapsto (\val (z), \val (w))$, 
which is a homomorphism. 

\begin{defi}
The closure
\[
	T_F
	:= \mathrm{Closure}({\VAL(\{p \in (K^*)^2 ; F(p)=0\})})
	\subset \mathbb{R}^2
\]
with usual topology on $\mathbb{R}^2$ 
is called the \textit{(plane) tropical curve defined by $F$}.   
\end{defi}

Remark that a tropical curve 
$T_F$ 
has a structure of a 
$1$-dimensional simplicial complex, cf.\cite{Mik_enumerative}. 
We call 
a $1$-simplex an \textit{edge} 
and 
a $0$-simplex a \textit{vertex} 
as usual.

Let $F$ be a polynomial over $K$. 
We introduce the Duality Theorem 
which gives a correspondence 
between 
the tropical curve $T_F$ defined by $F$ 
and the Newton polytope $\Delta_F$ of $F$. 
Let 
$\Delta_{\nu}(F)$ 
be the 3-dimensional polygon defined by 
\[
	\Delta_{\nu}(F)
	:=\mathrm{Conv}
	\{ (i,j,-\val (c_{ij}))\in \mathbb{R}^2 \times \mathbb{R}
	; (i,j) \in \mathrm{Supp}(F) \}
	\subset \mathbb{R}^3
\]
and 
$\nu_{F}: \Delta_F \to \mathbb{R}$
be the function defined by 
\[
	(i,j) \longmapsto 
	\min\{ x \in \mathbb{R} ; (i,j,x) \in \Delta_{\nu}(F) \}, 
\]
which is a continuous piecewise linear convex function. 
We can get the following three kinds of sub-polytopes of 
$\Delta_F$ from $\nu_{F}$: 
\begin{itemize}
\item 
linearity domains of $\nu_F$: 
$\Delta_1, \cdots ,\Delta_N$,
\item 
1-dimensional polytopes: 
$\Delta_i \cap \Delta_j \neq \emptyset$ and $\neq \{ pt \}$,
\item 
0-dimensional polytopes: 
$\Delta_{i_1} \cap \Delta_{i_2} \cap \Delta_{i_3} \neq \emptyset$, 
\end{itemize}
where a \textit{linerity domain} of $\nu_F$ 
means a maximal sub-polytope $R$ of the domain $\Delta_F$
such that the restriction $\nu_F|_{R}$ is an affine linear function.  
These polytopes give a 
subdivision of 
$\Delta_F$, 
which we denote by 
$\subd(F)$. 
In particular, 
we call 
a 1-dimensional  
and 
a 0-dimensional polytope 
a \textit{vertex} 
and 
a \textit{edge} 
of 
$\subd(F)$, respectively.

For edges of a tropical curve, 
a certain weight 
$ w:(\text{edges of }T) \to \mathbb{N} $ 
is defined by directional vectors of edges canonically. 
We omit the definition since 
we don't use it in this paper. 

We can find the following claim in \S 2.5.1 of \cite{IMS}. 
A proof in general dimension can be found in \cite{Mik_enumerative}. 

\begin{thm}[Duality Theorem]\label{DualThm}\it
The subdivision 
$\subd(F)$ 
is dual to the tropical curve
$T_F$ 
in the following sense: \\
(1)\;
the components of 
$\mathbb{R}^2 \setminus T_F$ 
are in 1-to-1 correspondence 
with the vertices of the subdivision 
$\subd(F)$, \\
(2)\;
the edges of 
$T_F$ 
are in 1-to-1 correspondence 
with the edges of the subdivision 
$\subd(F)$ 
so that an edge 
$E \subset T_F$ 
is dual to an orthogonal edge of the subdivision 
$\subd(F)$, 
having the lattice length equal to 
$w(E)$, 
which is a weight of 
$E$, \\
(3)\;
the vertices of 
$T_F$ 
are in 1-to-1 correspondence with 
the polytopes 
\[ \subd(F) : \Delta_{1}, \dots, \Delta_{N} \]
so that the valency of a vertex of 
$T_F$ 
is equal to the number of sides of the dual polygon.
\end{thm}

We call the subdivision 
$\subd(F)$ 
the \textit{dual subdivision} 
of $T_F$. 
By Theorem~\ref{DualThm}, 
we can regard a tropical curve 
as a dual subdivision of the Newton polytope 
of its defining polynomial. 

\section{Main Results}
In this section, 
we first introduce the precise definition of 
tropical sub-curves which we mentioned in the introduction. 
Let 
$F \in K[x,y]$ 
be a polynomial over $K$ 
and 
\[ \subd (F): \Delta_1 , \dots , \Delta_N \]
be the dual subdivision of $T_F$. 
Because of the structure theorem in tropical geometry, 
a tropical hypersurface has the structure of a polyhedral complex.
In particular, 
a plane tropical curve is an embeded plane graph in $\mathbb{R}^2$.

Let $[u,v]$ denote the edge of 
the tropical curve $T_F$ 
whose endpoints are 0-cells $u$ and $v$. 
Set $[u,v)=[u,v]\setminus \{v\}$. 
We allow that one of the endpoints is at $\infty$. 
In this case, 
the other endpoint is contained in the 0-cells of the curve.  
Note that $[u,\infty]=[u,\infty)$. 

Let 
\[ 
\subd' (F) : \Delta_{k_1} , \dots , \Delta_{k_m}  
\]
be a subset of $\subd(F)$ 
such that 
$\bigcup \mathrm{Sd}'(F) \setminus \mathrm{Sd}^{[0]}(F)$ 
is connected, 
where 
$\mathrm{Sd}^{[0]}(F)$ 
is the set of vertices of 
$\mathrm{Sd}(F)$.
Let $\Delta'$ be 
a sub-polyhedron of 
$\Delta_{F}$ 
given as the union of 
$\subd'(F)$. 

Let 
$V=\{ v_1 , \dots ,v_N \}$ 
and
$E=\{ [u,v]\ ;\ u,v \in V \}$
be the set of vertices and edges of 
$T_F$ 
respectively. 

\begin{defi}\label{resttrop}
A subset of 
$T_F$ 
is called the 
\textit{tropical sub-curve 
with respect to 
$\Delta'$} 
if it has the structure of 
the metric (open) sub-graph $(V',E')$ 
of the tropical curve $T_F$ 
which satisfies the following conditions: \\
(1)\; 
the set of vertices 
$V' \subset V$ 
is given by 
$\{ v_{k_1},\dots ,v_{k_m} \}$, \\ 
(2)\;
the set of edges $E'$ is given by the following manners: 
for each $[u,v] \in E$,
\begin{itemize}
\item[(i)] 
if $u,v \in V'$ then $[u,v] \in E'$,
\item[(ii)]
if $v = \infty$ and $u \in V'$ 
then $[u,v] \in E'$,
\item[(iii)] 
if $u \in V'$ and $ v \in {V} \setminus V'$ 
then $[u,v') \in E'$, 
where $v'$ is taken as the middle point of 
$[u,v]$.  
\end{itemize} 
We denote the tropical sub-curve of 
$T_F$ 
with respect to 
$\Delta'$ 
by 
$T_F|_{\Delta'}$. 
\end{defi}

\begin{exa}\label{resttropex}
(1)\;
Let $F$ be a polynomial over $K$ given by 
\begin{align*} 
F= 
& 1+tz+tw+t^3z^2+t^2zw+t^3w^2+t^6z^3+t^4z^2w+t^4zw^2+t^6w^3 \\
& +t^{10}z^4+t^7z^3 w +t^{12}z^2w^2 +t^7zw^3 + t^{10}w^4 
  +t^{15}z^5+t^{15}w^5. 
\end{align*}
The Newton polytope 
$\Delta_F$ of $F$ 
is 
$\mathrm{Conv}\{ (0,0),(0,5),(5,0) \}$. 
See on the left in Figure~1. 
This polynomial $F$ is $F_f$ in Theorem~\ref{MainThm} 
for the singularity of 
$x^5 +x^2 y^2 +y^5$ 
at the origin. 
The polyhedron 
$\Delta'$ 
in the figure is 
$\Gamma_-(f)$ 
for the singularity 
$(f,0)$. 
The tropical sub-curve 
$T_F|_{\Gamma_{-}(f)}$ 
with respect to 
$\Gamma_-(f)$ is as shown on the right.
Since 
$\mu(f)= 11, v(T_F|_{\Gamma_{-}(f)})=6$ and $r(T_F|_{\Gamma_{-}(f)})=5 $,  
the equality 
$\mu(f) = v(T_F|_{\Gamma_{-}(f)})+r(T_F|_{\Gamma_{-}(f)})$ 
in Theorem \ref{MainThm} is verified.

\begin{figure}[htbp] \label{fig}
\includegraphics[scale=0.6]{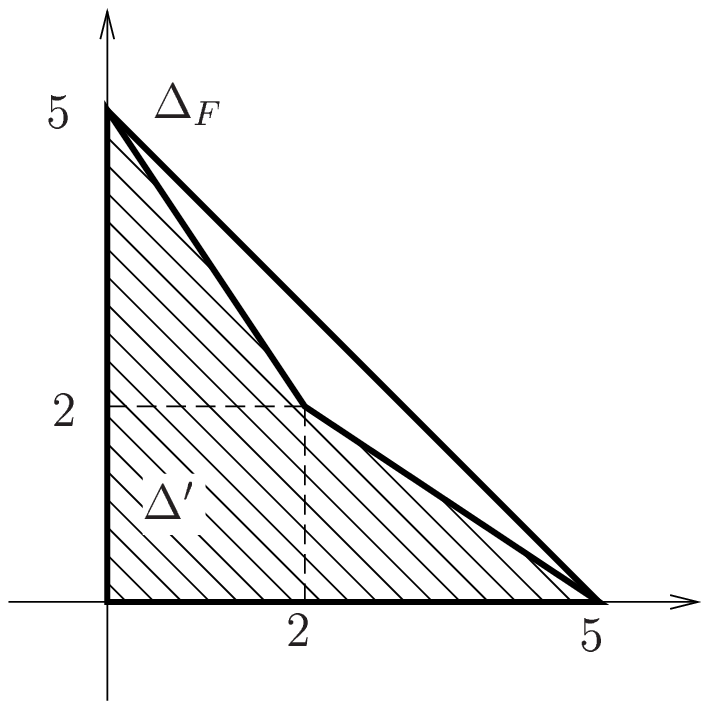}
\hspace{15mm}
\includegraphics[scale=0.5]{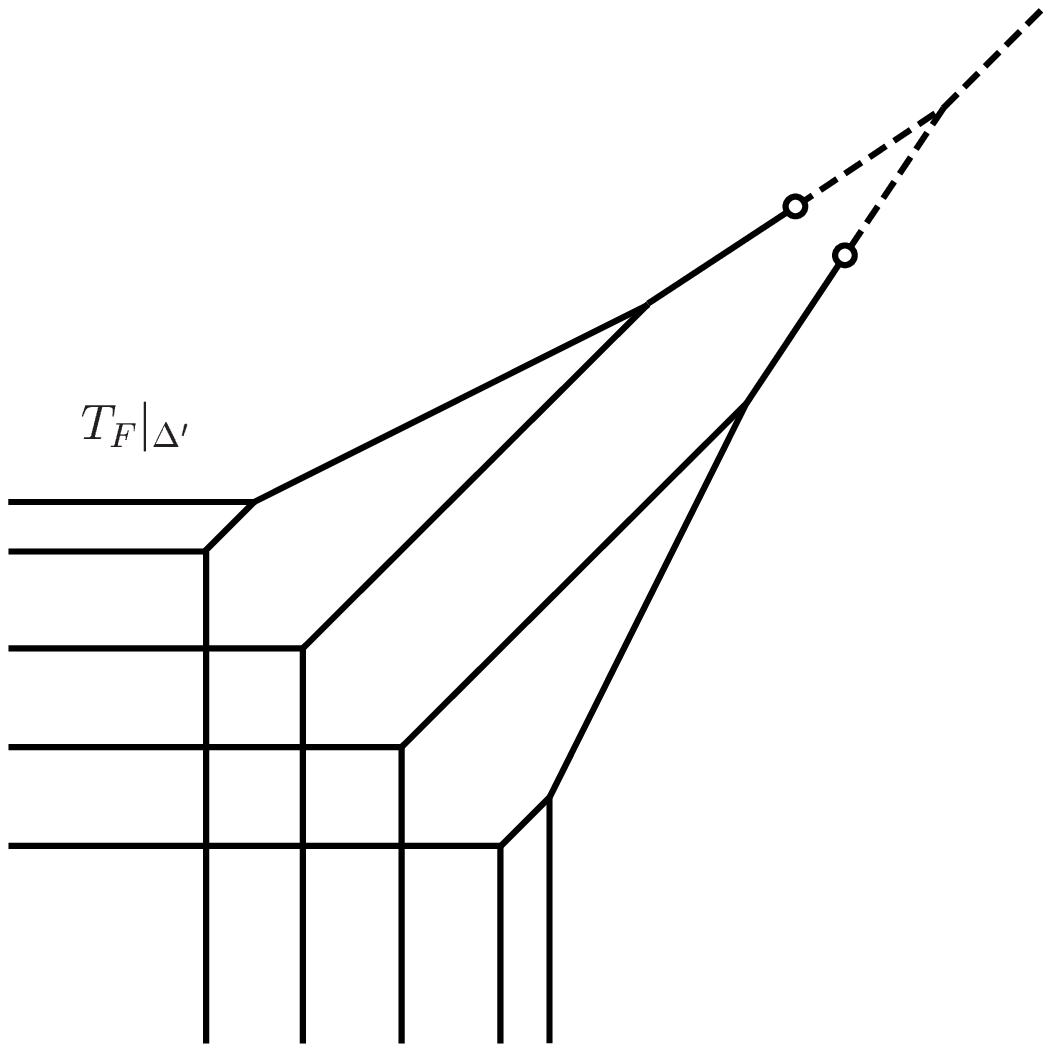}
\caption{$\Delta_F$ and $T_F|_{\Gamma_-(f)}$ in Example \ref{resttropex} (1).}
\end{figure}

(2)\; 
Let $F$ be a polynomial over $K$ given by  
\[
F= 1+tz+tw+ t^3z^2+t^2zw+t^3w^2+t^6w^3.
\]
The Newton polytope $\Delta_F$ of $F$ is 
$\mathrm{Conv}\{ (0,0),(2,0),(0,3) \}.$
See on the left in Figure~2. 
This polynomial $F$ is $F_f$ in Theorem \ref{MainThm} 
for the singularity of $x^2+y^3$ at the origin. 
The polyhedron 
$\Delta_F$ 
in the figure is 
$\Gamma_-(f)$ 
for the singularity $(f,0)$. 
The tropical sub-curve 
$T_F|_{\Gamma_{-}(f)}$ 
with respect to 
$\Gamma_-(f)$ is as shown on the right.
Since 
$\mu(f)=2 , v(T_F|_{\Gamma_{-}(f)})=1$ and $r(T_F|_{\Gamma_{-}(f)})=1 $,  
the equality 
in Theorem \ref{MainThm} holds.

\begin{figure}[htbp] \label{fig}
\includegraphics[scale=0.8]{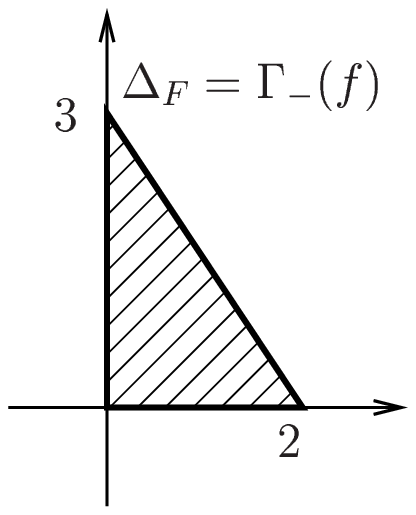}
\hspace{15mm}
\includegraphics[scale=0.7]{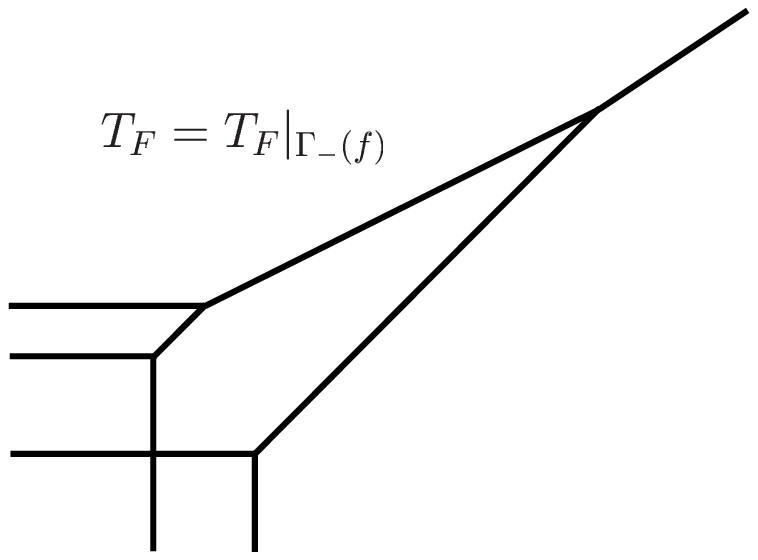}
\caption{$\Delta_F$ and $T_F$ in Example \ref{resttropex} (2).}
\end{figure}

\end{exa}

Suppose that $f$ is convenient. 
For the lattice points 
$(i,j) \in \Gamma_-(f) \cap \mathbb{Z}^2$, 
we define a map 
$\nu_f|_{
\Gamma_-(f) \cap \mathbb{Z}^2}:
\Gamma_-(f) \cap \mathbb{Z}^2 \to \mathbb{R}$ 
by 
\[ \nu_f(i,j)=a_0 + a_1+ \cdots +a_i+b_0+b_1+ \cdots +b_j \] 
where 
$\{ a_k \}_{k \in \mathbb{N}}, \{b_k\}_{k \in \mathbb{N}}$ 
are non-negative strictly increasing sequences 
of integers. 
We then extend it to the whole domain 
$\Gamma_-(f)$ 
as a continuous piecewise linear function and 
obtain a map 
$\nu_f:\Gamma_-(f) \to \mathbb{R}$. 
Taking sufficiently large values for 
$\nu_f$ 
at the lattice points of the Newton boundary of $f$, 
we may assume that 
the other sub-polytopes are triangles with area $1/2$.

\begin{defi}
We call the subdivision of 
$\Gamma_{-}(f)$ 
defined as above the 
\textit{special subdivision of $\Gamma_{-}(f)$} 
and each square in this subdivision
the \textit{special square}. 
\end{defi}

\begin{lem}\label{lem1}\it
Let $p,q \in \mathbb{N}$ be coprime integers.
The number of special squares in the special subdivision of
$\Delta_{(p,q)}=\mathrm{Conv}\{ (0,0),(p,0),(0,q) \} \subset \mathbb{R}^2$ 
is 
${(p-1)(q-1)}/{2}$. 
\end{lem}

\begin{proof}
Let 
$\hat{\Delta}$ 
be the rectangle given by 
\[
\hat{\Delta}
=\mathrm{Conv}\{ (0,0),(p,0),(0,q),(p,q) \}\subset \mathbb{R}^2.
\]
We consider the special subdivision of 
$\hat{\Delta}$. 
We decompose it into $p$ vertical rectangles 
\[ 
\hat{\Delta}_i =
([i,i+1] \times \mathbb{R})\cap \hat{\Delta} \subset \hat{\Delta},
\ \ \ i=0, \dots ,p-1. 
\]
The special subdivision of 
$\hat{\Delta}$ 
induces a special subdivision of 
each 
$\hat{\Delta}_i$.  
Let 
$\ell$ 
be the segment connecting $(p,0)$ and $(0,q)$. 
We denote by $I$ 
the number of special squares in $\hat{\Delta}$ which intersect $\ell$.
Similarly, we denote by $I_i$ 
the number of special squares in $\hat{\Delta}_i$ which intersect $\ell$.
Obviously $I = \sum_{i=0}^{p-1}I_i$.

Let $\lambda$ be the number of special squares of $\Delta_{(p,q)}$.
Notice that $\lambda = \frac{1}{2}(pq-I) $ since $pq=2 \lambda +I$.
Thus it is enough to show 
$I=p+q-1$. 
Without loss of generality, 
we may assume $p < q$.
Let $k$, $l$ be integers such that 
$ q=pk+l$ and $0<l<p$. 
The segment $\ell$ can be denoted as 
\[ 
\Bigl( x,-\frac{q}{p}x+q \Bigr) =\Bigl( x,(pk+l-kx) -\frac{l}{p}x \Bigr)
,\ \ \ x \in [0,p]. 
\]
Set $\xi(x)=\frac{l}{p}x$. 
Then, 
$I_i$ is calculated as  
\begin{align*}
I_i & = pk+l-nk-\lfloor \xi(i) \rfloor - \Bigl\{ pk+l-(i+1)k -\lfloor \xi(i+1) \rfloor -1 \Bigr\}\\
    & = k+1 - \lfloor \xi(i) \rfloor + \lfloor \xi(i+1) \rfloor 
\end{align*}
for $i=1 , \dots , p-2$ and 
\[ I_0 = I_{p-1}= k+1, \]
where 
$\lfloor \alpha \rfloor$ 
means the largest integer not greater than $\alpha \in \mathbb{R}$. 
Thus, we obtain
\begin{align*}
I &= \sum_{i=0}^{p-1}I_i 
  = I_0 + (p-2)(k+1)+\lfloor \xi(p-1) \rfloor - \lfloor \xi(1) \rfloor +I_{p-1} \\
 &= pk +p +\lfloor \frac{l}{p}(p-1) \rfloor 
  = p+q-1. 
\end{align*}
\end{proof}

\begin{proof}[Proof of Theorem \ref{MainThm}]
Choose a polynomial $F$ such that 
the Newton polytope $\Delta_F$ coincides with 
$\mathrm{Conv}(\Gamma_-(f))$. 
To decide coefficients of $F$, 
we take the convex function
$\nu : \Delta_F \cap \mathbb{Z}^2 \to \mathbb{R}$
as a linear extension of 
$\nu_f$ 
used in the definition of the special subdivision of 
$\Gamma_-(f)$,  
and define $F$ as 
the \textit{patchworking polynomial} defined by $\nu$, 
that is, 
\[ 
F(z,w)
=\sum_{(i,j) \in {\Delta_F \cap \mathbb{Z}^2}}
t^{-\nu(i,j)}z^iw^j. 
\] 

In the rest of the proof, 
we check 
$T_F|_{\Gamma_-(f)}$ 
satisfies the equality in the assertion.  
To calculate the number of special squares, 
we decompose 
$\Gamma_{-}(f)$ 
into two sub-polyhedrons as follows. 
Let $p,q \in \mathbb{N}$ be coprime integers. 
We denote the intersection points of 
the Newton boundary $\Gamma(f)$ of $f$ 
and the lattice by 
\[ \Gamma(f) \cap \mathbb{Z}^2 =\{ (0,q),(P_1 , Q_1),\dots ,(P_{n-1},Q_{n-1}),(p,0) \}, \]
where $0 < P_1 < \dots < P_{n-1} <p$.
We set $P_0 = 0, P_n=p, Q_0=q, Q_n =0$ and define  
\[ p_i=| P_i - P_{i-1} |,\ q_i = |Q_i - Q_{i-1}|, \ \ \ i=1,\dots,n. \]
Notice that $p=p_1 + \dots +p_n$ and $\ q = q_1 + \dots + q_n$.
For $i=1,\dots,n$,
we define the subset $\Delta_i$ of $\Gamma_{-}(f)$ as  
\[ \Delta_i= \mathrm{Conv} \{ (P_{i-1},Q_{i-1}),(P_{i-1},Q_i),(P_i , Q_i) \} = \Delta_{(p_i , q_i)} \]
and 
\[
 \Xi_1 := \bigcup_{i=1}^{n} \Delta_{i} \subset \Gamma_{-}(f), \;\;\;
 \Xi_2 := \mathrm{Closure} \bigl( \Gamma_{-}(f) \setminus \Xi_1 \bigr) \subset \Gamma_{-}(f). 
\]
Then, $\Gamma_{-}(f)$ decomposes as 
$ \Gamma_{-}(f)=\Xi_1 \cup \Xi_2$. 
For $i=1,2$, 
we denote by $|\Xi_i|$ 
the number of special squares 
contained in the special subdivision of $\Xi_i$ 
induced by that of $\Gamma_{-}(f)$.
Then, using Lemma~\ref{lem1}, we have
\begin{align*}
|\Xi_1| &= \sum_{i=1}^n \frac{1}{2}(p_i -1)(q_i-1), \\
|\Xi_2| &= \sum_{i=1}^{n-1}p_i \cdot (q_i + \dots + q_n) \\
&=\sum_{i=1}^{n-1}p_i \cdot \{ q -(q_1 + \dots + q_{i-1}) \}
          = \mathrm{Vol}(\Xi_2) .
\end{align*}

Next we will show the following equalities: 
\begin{align}\label{numofsuq}
& | \Xi_1 | + | \Xi_2 | = v(T_F|_{\Gamma_-(f)}), \\ 
& | \Xi_1 | + | \Xi_2 | -(n-1) = r(T_F|_{\Gamma_-(f)}).
\end{align}
By Theorem~\ref{DualThm}, 
the correspondence between 
subdivisions and tropical curves, introduced in Theorem~\ref{DualThm}, 
gives a 
1-to-1 correspondence of 
parallerograms and 4-valent vertices.
In our case, 
any 4-valent vertex corresponds to a special square. 
Thus, the number of special squares, $|\Xi_1|+|\Xi_2|$,
coincides with the number of 4-valent vertices of $T_F|_{\Gamma_-(f)}$. 
Hence equality (1) holds. 

We prove the other equality. 
There is a 1-to-1 correspondence between
\[ 
\{ 
\text{special square $\boxtimes$ contained 
in special subdivision of }\Gamma_{-}(f) 
\ ;\ V(\boxtimes) \cap \Gamma(f) = \emptyset 
\} 
\]
and
\[ \mathrm{int}(\Gamma_{-}(f)) \cap \mathbb{Z}^2, \]
where 
$V(\boxtimes)$ 
is the set of vertices of 
a special square $\boxtimes$ in $\Gamma_{-}(f)$.
Moreover, by the Duality Theorem~\ref{DualThm} of tropical curves, 
we have a 1-to-1 correspondence between 
the bounded regions contained 
in the complement of 
$T_F|_{\Gamma_-(f)} \subset \mathbb{R}^2$ 
and 
the interior lattice points 
$\mathrm{int}(\Gamma_{-}(f)\cap \mathbb{Z}^2)$ 
in $\Gamma_-(f)$. 
Since 
\begin{align*}
\sharp 
&\{ 
\text{special square $\boxtimes$ contained in special subdivision of }\Gamma_{-}(f) 
; V(\boxtimes) \cap \Gamma(f) \neq \emptyset \} \\
&= n-1, 
\end{align*}
we get 
\begin{align*}
r(T_F|_{\Gamma_-(f)}) 
& = \sharp ( \mathrm{int}(\Gamma_{-}(f)) \cap \mathbb{Z}^2 ) \\
& = | \Xi_1 | + | \Xi_2 | - (n-1).
\end{align*}
Thus equality (2) holds.

Set 
$L=\sum_{i=1}^n p_i q_i$. 
From equality (1), 
we get $L = 2 |\Xi_1| + (p+q) -n$ 
as
\begin{align*}
|\Xi_1| 
&= \sum_{i=1}^n \frac{1}{2}(p_i -1)(q_i-1) 
= \frac{1}{2} \sum_{i=1}^n (p_i q_i -p_i - q_i +1) \\
&= \frac{1}{2} \bigl\{ L -(p+q) +n \bigr\}. 
\end{align*}

For the Milnor number $\mu(f)$, 
we use Kouchnirenko's formula in \cite{Kouch}: 
\[ \mu(f) =2V_2 - V_1 +1, \]  
where 
\[ V_2 = \frac{L}{2}+ |\Xi_2|,\ V_1 = p+q. \]
Thus, we get 
\begin{align*}
\mu(f) 
&= 2 V_2 -V_1 +1 \\
&= L +2| \Xi_2 | -(p+q) +1 \\
&= \bigl\{ 2 |\Xi_1| + (p+q) -n \bigr\} +2|\Xi_2| -(p+q) +1 \\
&=  v(T_F|_{\Gamma_-(f)}) + r(T_F|_{\Gamma_-(f)}).
\end{align*}
\end{proof}

\begin{cor}\label{Cor}\it
Let $F:=F_f$ be a polynomial obtained in Theorem~\ref{MainThm}. 
Then the number $\delta(f)$ of double points of $(f,0)$ coincides with 
$v(T_F|_{\Gamma_-(f)})$. 
\end{cor}

\begin{proof}
In \cite{Oka}, 
we have 
$ r(f) = \sharp (\mathbb{Z}^2 \cap \Gamma(f)) -1 $, 
where $r(f)$ is the number of local irreducible components of $f$ at $0$.
We also have 
\begin{align*}
\mu(f) 
&= 2(| \Xi_1 | + | \Xi_2 | ) - \bigl\{ \sharp (\mathbb{Z}^2 \cap \Gamma(f)) -2 \bigr\} \\
&= 2 v(T_F|_{\Gamma_-(f)}) - \bigl\{ \sharp (\mathbb{Z}^2 \cap \Gamma(f)) -2 \bigr\} 
\end{align*}
from the argument in the proof of Theorem \ref{MainThm}. Thus 
\begin{align*}
2\delta(f)
& =\mu(f) + r(f) -1 \\
& =2 v(T_F|_{\Gamma_-(f)}) - \bigl\{ \sharp (\mathbb{Z}^2 \cap \Gamma(f)) -2 \bigr\} + \sharp (\mathbb{Z}^2 \cap \Gamma(f)) -1 -1 \\
& =2 v(T_F|_{\Gamma_-(f)}).
\end{align*}
\end{proof}

\begin{rem}
As in \cite{AC, GZ}, 
we can obtain the intersection form of vanishing cycles 
from the immersed curve of a real morsification.  
To study the intersection form in our tropical curve  
we need to fix 
``framings" 
on edges of the curve 
in $\Gamma_{-}(f)$, 
though we don't have any good way to see 
these ``framings".
\end{rem}


\begin{thebibliography}{99}

\bibitem{AC}
 N.A'Campo, 
 \textit{Singularities and Related Knots}, Note by W.Gibson and M.Ishikawa, 
 unpublished.

\bibitem{Ganor}
 Y.Ganor, 
 \textit{Enumerating cuspidal curves on toric surfaces}, 
 arXiv:1306.3514v1.

\bibitem{GZ}
 S.M.Gusein-Zade,
 \textit{Intersection matrices for certein singularities of functions of two variables},
 Funct. Anal. Appl. \textbf{8} (1974), 10--13.

\bibitem{IMS}
 I.Itenberg, G.Mikhalkin, E.Shustin, 
 \textit{Tropical Algebraic Geometry}, Second Edition, 
 Oberwolfach Seminars Volume 35, 2009.

\bibitem{Kouch}
 A.G.Kouchnirenko, 
 \textit{Poly\'edres de Newton et nombres de Milnor}, 
 Invent. Math. \textbf{32} (1976), 1--31.

\bibitem{MMS}
 H.Markwig, T.Markwig, E.Shustin,
 \textit{Tropical curves with a singularity in a fixed point},
 Manuscripta Math. \textbf{137} (2012), no.3-4, 383--418.

\bibitem{Mik_enumerative}
 G.Mikhalkin, 
 \textit{Enumerative Tropical Algebraic Geometry in $\mathbb{R}^2$}, 
 J. Amer. Math. Soc. \textbf{18} (2005), 313--377.

\bibitem{Oka}
 M.Oka, 
 \textit{Introduction to Plane Curve Sigularities. Toric Resolution Tower and Puiseux Pairs},   
 Progress in Mathematics \textbf{283} (2009), 209--245.

\bibitem{Shu2}
 E.Shustin,
 \textit{Critical points of real polynomials, subdivisions of Newton polyhedra and topology of real algebraic hypersurfaces}, 
Amer. Math. Soc. Transl. (2), \textbf{173} (1996), 203--223.

\bibitem{Shustin}
 E.Shustin,
 \textit{A tropical approach to enumerative geometry},
 Algebra i Analiz, \textbf{17} Issue 2 (2005), 170--214.

\end{thebibliography}
\end{document}